\documentclass[12pt]{amsart}
\usepackage{amscd}
\usepackage{mathrsfs}
\usepackage{a4wide}
\usepackage{color}
\usepackage{amssymb,amsmath,amsthm,amsfonts,latexsym}
\usepackage[utf8x]{inputenc}
\usepackage{bbm} 

\usepackage{amssymb}

\usepackage[cmtip,all]{xy}

\newtheorem{theorem}{Theorem}[section]
\newtheorem{lemma}[theorem]{Lemma}
\newtheorem{conjecture}[theorem]{Conjecture}
\newtheorem{proposition}[theorem]{Proposition}

\newtheorem{observation}[theorem]{Observation}
\newtheorem{corollary}[theorem]{Corollary}
\newtheorem{question}[theorem]{Question}

\theoremstyle{definition}
\newtheorem{definition}[theorem]{Definition}

\theoremstyle{remark}

\numberwithin{equation}{section}

\newcommand{\Fra}{Fra\"\i ss\'e}
\newcommand{\Flim}{\operatorname{Flim}}
\newcommand{\Gur}{Gurarij}
\newcommand{\Aut}{\operatorname{Aut}}
\newcommand{\Age}{\operatorname{Age}}
\newcommand{\Cee}{{\mathcal C}}
\newcommand{\Gee}{{\mathcal G}}
\newcommand{\Kat}{{Kat\v etov}}
\newcommand{\Met}{{\mathcal M}}
\newcommand{\U}{{\mathbb U}}

\begin{document}

\title{A rigid Urysohn-like metric space}

\author{Jan Greb\'ik}
\address{Institute of Mathematics, Czech Academy of Sciences, 115 67 Prague, Czech Republic}
\curraddr{}
\email{}
\thanks{This work is part of the author's MSc thesis written under the supervision of Wies{\l}aw Kubi\'s.
Research supported by GA\v CR project 15-34700L and partially supported by
MOBILITY project 7AMB15AT035 (RVO:67985840). }

\subjclass[2010]{Primary 03C50, 
05C63. 
}

\date{}

\dedicatory{}

\commby{Mirna D\v zamonja}

\begin{abstract}
Recall that the \emph{Rado graph} is the unique countable graph that realizes all one-point extensions of its finite subgraphs. The Rado graph is well-known to be universal and homogeneous in the sense that every isomorphism between finite subgraphs of $R$ extends to an automorphism of $R$.

We construct a graph of the smallest uncountable cardinality  $\omega_1$ which has the same extension property as $R$, yet its group of automorphisms is trivial. We also present a similar, although technically more complicated, construction of a complete metric space of density $\omega_1$, having the extension property like the Urysohn space, yet again its group of isometries is trivial. This improves a recent result of Bielas.
\end{abstract}

\maketitle

\section{Introduction}

Recall that a structure $M$ is \emph{homogeneous} if every isomorphism between finitely generated substructures of $M$ extends to an automorphism of $M$.
A structure $M$ is \emph{$\omega$-saturated} if for every finitely generated structures $A \subseteq B$ every embedding of $A$ into $M$ extends to an embedding of $B$ into $M$. Of course, in this definition only structures from a fixed class $\Cee$ are considered. 
Finally, a structure $M$ is $\Cee$-\emph{universal} if every $X \in \Cee$ embeds into $M$.
A countably generated homogeneous $\Cee$-universal structure that also belongs to $\Cee$ is called the \emph{\Fra\ limit} of $\Cee$ (or, more precisely, of the class of finitely generated structures that are in $\Cee$). The key fact needed for the existence of a \Fra\ limit is the \emph{amalgamation property} saying that for every two embeddings $e_1 \colon A \to B_1$, $e_2 \colon A \to B_2$, where $A, B_1, B_2 \in \Cee$ are finitely generated, 
there exist embeddings $f_1 \colon B_1 \to C$, $f_2 \colon B_2 \to C$ with $C \in \Cee$, making the diagram
$$\xymatrix{
B_1 \ar[rr]^{f_1} & & C \\
A \ar[u]^{e_1} \ar[rr]_{e_2} & & B_2 \ar[u]_{f_2}
}$$
commuting.
Note that in case of relational languages (that is, languages with relation symbols only) finitely/countably generated structures are finite/countable.
One of the typical and well explored classes is the class $\Gee$ of countable graphs.
Its \Fra\ limit is the \emph{Rado graph}.

Another, formally not fitting into the above framework, is the class $\Met$ of complete separable metric spaces. Here, being finitely generated still means finite, yet a countably generated substructure is a separable closed subspace. All other concepts are the same as before.
The ``\Fra\ limit'' of $\Met$ is the \emph{Urysohn space} $\U$, constructed by Urysohn \cite{Ury} in his last work, published in 1927.
A rational version of $\Met$, denoted by $\mathbb Q\Met$, is the class of countable metric spaces with rational distances. This fits into the model theoretic setting and its \Fra\ limit is the \emph{rational Urysohn space} $\mathbb Q\U$, which is the unique countable $\omega$-saturated rational metric space. Its metric completion is the Urysohn space $\U$.

Our aim is to present two examples of $\omega$-saturated structures with trivial automorphism groups (such structures are called \emph{rigid}).
Namely, we construct an $\omega$-saturated rigid graph of cardinality $\omega_1$ and a rigid $\omega$-saturated complete metric space of density $\omega_1$. The second examples is an improvement of a recent result of Bielas~\cite{Bie}, who constructed an example with the same properties, however its density is large (strictly above the continuum).

In order to construct the announced examples, we prove the existence of an embedding $e \colon M \to M$, where $M$ is either the Rado graph or the Urysohn space $\U$, such that no non-trivial automorphism of $e[M]$ extends to $M$. In the case of graphs such a result has already been proved by Imrich, Klav\v zar, and Trofimov~\cite{Imr}.

Our results show that uncountable or non-separable $\omega$-saturated structures can have properties very far from being homogeneous.
This gives rise to a question whether there exist uncountable (or non-separable) $\omega$-saturated structures that are homogeneous with respect to its finitely generated substructures.
It turns out that the answer is affirmative as long as the class admits a so-called \emph{Kat\v etov functor}.
In that case it is not hard to see that for each uncountable cardinal $\kappa$ there exists a homogeneous $\omega$-saturated structure of size $\kappa$.
We sketch the arguments in the next section. For precise definitions and results on Kat\v etov functors we refer to~\cite{KubMas}.

\section{Kat\v etov functors and $\omega$-saturated structures}

Let $\Cee$ be a fixed class of countably generated structures, where in case of metric spaces ``countably generated'' means ``closed separable''.
We denote by $\Flim(\Cee)$ the \Fra\ limit of $\Cee$, namely, the unique countably generated (complete separable, in case of metric spaces) structure $L$ that is homogeneous and $\Cee$-universal.
It is well-known that $\Flim(\Cee)$ exists if and only if $\Cee$ has the joint embedding property (every two finitely generated structures are isomorphic to substructures of some $C \in \Cee$), the amalgamation property, and contains countably many isomorphic types of finitely generated structures.
In the case of metric spaces the last condition is not satisfied, although the Urysohn space still shares all the properties of model-theoretic \Fra\ limits.
For general theory of \Fra\ limits we refer to \cite{Hod}, for category-theoretic generalizations see \cite{Kub}.

Recall that the \emph{age} of a structure $X$, denoted by $\Age(X)$, is the class of all finitely generated structures isomorphic to substructures of $X$. Clearly, $\Cee = \Age(\Flim(\Cee))$, as long as $\Flim(\Cee)$ exists.

\begin{definition}
We say that structure (not necessarily countable) X is \emph{\Fra-like} for $\mathcal{C}$ if $\Age(X)=\mathcal{C}$ and it is $\omega$-saturated.
In the particular classes of countable graphs and complete separable metric spaces, we shall say \emph{Rado-like} and \emph{Urysohn-like} instead of \Fra-like.
\end{definition}

From now on we will be mostly interested in \Fra-like structures of cardinality $\omega_1$. It can be proved that every such structure is the colimit of a continuous transfinite chain of length $\omega_1$ of the following form:
$$\Flim(\mathcal{C})\hookrightarrow \Flim(\mathcal{C}) \hookrightarrow \dots \hookrightarrow\Flim(\mathcal{C}) \hookrightarrow \dots.$$
The embeddings in this sequence can be of course completely arbitrary; in typical cases  there are continuum many possibilities for an embedding $\Flim(\Cee) \hookrightarrow \Flim(\Cee)$.
Continuity of the chain simply means that the structures at limit steps are colimits of the smaller ones. Note that the colimit of a chain of first-order structures is simply its union, while the colimit of a chain of complete metric spaces (with isometric embeddings) is the completion of its union.

Assume that we have such structure $X$, we say that it is \emph{given by} a sequence $(\Flim(\mathcal{C}),e_i)_{i<\omega_1}$, where $e_i$ are embeddings as above (more precisely, $e_i$ is the embedding of $i$th copy of $\Flim(\Cee)$ into $(i+1)$st copy of $\Flim(\Cee)$). We will use the obvious notation $e_i^j$, denoting the embedding of $i$th structure of the chain into the $j$th structure. It is straightforward to see that for every automorphism $\alpha$ of $X$ there is a closed and unbounded set of indices $C \subseteq \omega_1$ such that $\alpha$ is invariant on $\Flim(\Cee)_\alpha$ for every $\alpha \in C$, where $\Flim(\Cee)_\alpha$ denotes the $\alpha$th copy of $\Flim(\Cee)$ in the chain.

\begin{definition}
To every $e:\Flim(\mathcal{C})\to \Flim(\mathcal{C})$ we assign $G_e\le \Aut(\Flim(\mathcal{C}))$ such that $\alpha\in \Aut(\Flim(\mathcal{C}))$ is in $G_e$ iff there is $\beta\in \Aut(\Flim(\mathcal{C}))$ such that the following diagram commutes

$$\xymatrix{
\Flim(\mathcal{C}) \ar[rr]^{\beta} & & \Flim(\mathcal{C}) \\
\Flim(\mathcal{C}) \ar[u]^{e} \ar[rr]_{\alpha} & & \Flim(\mathcal{C}) \ar[u]_{e}
}.$$
\end{definition}

We say that such $\beta\in \Aut(\Flim(\mathcal{C}))$ is \emph{invariant} over $e$ and $\alpha$ can be extended \emph{via} $e$.  We can define a subgroup $H_e\le \Aut(\Flim(\mathcal{C}))$ which consist of those elements which are invariant over $e$. There is a natural homomorphism $h:H_e\to G_e$ which is onto. To every \Fra-like structure $X$ given by a sequence $(\Flim(\mathcal{C}),e_i)_{i<\omega_1}$ we assign a tree $T_X$. Its elements are automorphisms of $\Flim(\mathcal{C})$ for all $i<\omega_1$ and the ordering is given by the relation of being invariant and can be extended i.e. $\alpha\ge \beta$ iff $\alpha\in \Aut(\Flim(\mathcal{C}))$, $\beta\in \Aut(\Flim(\mathcal{C}))$ and $\beta$ is an extension of $\alpha$ given by some $e_i^j$. In fact, $X$ has a non-trivial automorphism iff $T_X$ has a cofinal branch, different from the branch of identities.

A general approach by using so-called \emph{Kat\v etov functors} (see \cite{KubMas}) gives a sufficiant condition for the existence of homogeneus \Fra-like structure $X$. For example, graphs and metric spaces admit a Kat\v etov functor.
More generally, $\mathcal{C}$ has a Kat\v etov functor whenever it has push-outs in the category of homomorphisms, see \cite{KubMas}.

Roughly speaking, a Kat\v etov functor assigns to each structure $X$ a bigger structure $K(X) \supseteq X$ realizing all one-point extensions of finitely generated substructers of $X$. Furthermore, $K$ is a functor, which means that it assigns to each embedding its extension, and this assignment preserves identities and compositions.

\begin{proposition}\label{PropOne}
If there is a Kat\v etov functor then there is a non-trivial embedding $e: \Flim(\mathcal{C})\to \Flim(\mathcal{C})$ such that $G_e=\Aut(\Flim(\mathcal{C}))$.
\end{proposition}

\begin{theorem}
If there is a Kat\v etov functor then there is $X=(\Flim(\mathcal{A}),e_i)_{i<\omega_1}$ that is homogenous.
\end{theorem}
\begin{proof}
Take $e_i:=e$ from Proposition~\ref{PropOne}.
\end{proof}

We prove in the next sections that for graphs and metric spaces the opposite extreme possibility can hold as well.

\section{The Rado graph}
Let $\mathcal{G}$ be the \Fra\ class of finite graphs with embeddings. \Fra\  limit of $\mathcal{G}$ is called the Rado graph and we denote it by $\mathcal{R}$. Let just recall its well known characterisation.

\begin{quote}
For every $X,Y$ disjoint subsets of $\mathcal{R}$ there is an element $x\in \mathcal{R}$ which is connected with an edge to all elements in $X$ and not connected to all elements in $Y$.
\end{quote}

By general theory every countable graph having this property is isomorphic to $\mathcal{R}$. For better reading we call each \Fra\ -like graph Rado-like. Just for the sake of completeness recall the definition.

\begin{definition}
Let $X$ be a graph on $\omega_1$ vertices. Then it is Rado-like iff it is the colimit of a chain of the form $(\mathcal{R},e_i)_{i<\omega_1}$.
\end{definition}

As we have mentioned above, there is a Kat\v etov functor for graphs, therefore we have:

\begin{corollary}
There exists a homogeneous Rado-like graph.
\end{corollary}

Now we turn our attention to the opposite of Proposition \ref{PropOne} from the previous section.
The following result was proved by Imrich, Klav\v zar, and Trofimov. We present a slightly different proof, as similar techniques will be used later.

\begin{theorem}[\cite{Imr}]
There is an embedding $e:\mathcal{R}\to \mathcal{R}$ such that $G_e=\{id\}$.
\end{theorem}

\begin{proof}
Denote $\mathcal{R}=(V,E)$. Let $A_1\subseteq \mathbb{N}$ be infinite and fix the unique increasing enumeration on $A_1$ i.e. $A_1=\{a_1<a_2<...\}$. Fix an enumeration of all distinct ordered pairs of vertices from $\mathcal{R}$ and denote it by $\{u_1,u_2,...\}$. We add to $\mathcal{R}$ countably many vertices $V_1:=\{v_1,v_2,...\}$ and some edges such that it will be again isomorphic to $\mathcal{R}$. 

We want to add edges between vertices from $V_1$ in such a way that $d_{V_1}(v_i)=a_i$ ($d_{V_1}$ denotes the degree with respect to $V_1$) and there is a path $[v_1,v_2,...]$. This is always possible because the set $A_1$ is infinite. We proceed by induction. In the $n$-th step we already have that $d(v_n)=k$ for some $k<a_n$ so we add edges $\{a_n,a_m\}_{n<m\le n+a_n-k}$. 

Fix an enumeration $\{W_1,W_2,...\}$ of all disjoint ordered pairs of finite subsets of $V\cup V_1$ such that the following holds
\begin{itemize}
\item for every $j\in\mathbb{N}$ there is $i=0,1$ such that $W^i_j\cap V_1\not=\emptyset$, 
\item for every $n\in\mathbb{N}$ the following holds $(W^1_n\cup W^2_n)\cap \{v_n,v_{n+1},...\}=\emptyset.$
\end{itemize}
Such enumeration always exists. The construction will be as follows. In the $n$-th step we add $\{v_n,u^1_n\}$ as edge and we do not add $\{v_n,u^2_n\}$. We take care of the pair $W_n$ in such a way that $w_n$ is a witness from the original $\mathcal{R}$ for a vertex which is connected to all $W^1_n\cap V$, not to any of $W^2_n\cap V$ and it has not been used yet in any previous step. Finally we add these edges $\{\{w_n,v\}:v\in V_1\cap W^1_n\}$.

To complete the proof we have to show that we obtain a Rado graph and that it has the required properties. We can easily check that the Rado property is satisfied due to the construction. Assume now that we have an automorphism $\alpha\in G_e$. For every $i\not= j$ we have $d_{V_1}(v_i)\not=d_{V_1}(v_j)$ which implies that the induced graph on $V_1$ is rigid. That means that every extension has to be identity on $V_1$. Assume now that $\alpha$ moves at least one vertex i.e. $\alpha(x)=y$ and $x\not=y$. But this pair has a number, say $k$, and this means that $\{x,v_k\}$ is an edge which forces $\{\alpha(x)=y,\bar{\alpha}(v_k)=v_k\}$ to be an edge too, but $\{v_k,y\}$ is not an edge. This is a contradiction.
\end{proof}

\begin{theorem}
There exists a Rado-like graph that is rigid.
\end{theorem}
\begin{proof}
We use induction to build a sequence of graphs $\{\mathcal{R}_i\}_{i<\omega_1}$. We will denote vertices of $\mathcal{R}_i$ by $V_i$ and edges by $E_i$. Fix an almost disjoint family $\{A_i\}_{i<\omega_1}$ and fix the increasing ordering on each $A_i$ i.e. $A_i=\{a_{i,0}<a_{i,1}<...\}$. Let $\mathcal{R}_0:=\mathcal{R}$. Assume that we have constructed $\{\mathcal{R}_i\}_{i<\alpha}$ with the following properties:
\begin{itemize}
\item $\mathcal{R}_{\beta}=\bigcup_{i<\beta}\mathcal{R}_i$ for $\beta $ limit,
\item $\mathcal{R}_i\subseteq \mathcal{R}_{i+1}$ and $|V_{i+1}\setminus V_i|=\omega$ for all $i<\alpha$,
\item $\mathcal{R}_i\simeq \mathcal{R}$ for all $i<\alpha$,
\item if $v\in V_j$ and $w\in V_i$ where $0<j<i<\alpha$ then $\{v,w\}\not\in E_i$,
\item for every pair $\{v,w\}\in [\mathcal{R}_0]^2$ and every $i<\alpha$ there is $u\in V_{i+1}\setminus V_i$ such that $\{u,v\}\in E_{i+1}$ and $\{u,w\}\not\in E_{i+1}$. 
\end{itemize}
If $\alpha$ is limit then put simply $\mathcal{R}_{\alpha}=\bigcup_{i<\alpha}\mathcal{R}_{\alpha}$. If $\alpha$ is a succesor then $V_{\alpha}:=V_{\alpha-1}\cup \{v_{\alpha,0},v_{\alpha,1},...\}$. We need to add edges in such a way that the conditions above are satisfied. First, it is clear that we are not allowed to add edges in $V_{\alpha-1}$. We add edges between $\{v_{\alpha,0},v_{\alpha,1},...\}$ in such a way as in the proof of the previous theorem, so that $d(v_{\alpha,k})=a_{\alpha,k}$ and the induced graph on $\{v_{\alpha,0},v_{\alpha,1},...\}$ is connected. Finally we need to add edges between $\{v_{\alpha,0},v_{\alpha,1},...\}$ and $\mathcal{R}_0$ to make $\mathcal{R}_{\alpha}\simeq \mathcal{R}$. This can be done similarly as in the previous proof once we use the fact that for a pair of subgraphs $G_1,G_2\subseteq V_{\alpha}$ which we want to extend by one element connected to all vertices of $G_1$ and to none of $G_2$ we can always choose $u\in \mathcal{R}_0$ which has this property for the pair $G_1\cap \mathcal{R}_{\alpha-1},G_2\cap\mathcal{R}_{\alpha-1}$, due to the construction, namely the fourth condition in the above list.

To finish the proof it is enough to show that for all $j<i<\omega_1$ the group $G_{e_{j,i}}$ is trivial. Suppose it is not. There is a nontrivial $\alpha\in \Aut(\mathcal{R}_j)$ and $\beta\in \Aut(\mathcal{R}_i)$ such that $e_{j,i}\circ \alpha=\beta\circ e_{j,i}$. We prove that for $v\in V_{j+1}\setminus V_j$ it holds that $\beta(v)=v$. Indeed, otherwise $\beta(V_{j+1}\setminus V_j)=V_{k+1}\setminus V_k$ for some $j<k<i$ since $\{v_{j+1,0},v_{j+1,1},...\}=V_{j+1}\setminus V_j$ is connected and $\beta$ is an automorphism; but this means $A_{j+1}=\{d(v_{j+1,l})\}_{l<\omega}=d_(v_{k+1,l})_{l<\omega}=A_{k+1}$ which is a contradiction unless $j=k$ and $\beta\upharpoonright V_{j+1}\setminus V_j=id$. Together with the fourth condition from above we have that $\beta\upharpoonright \mathcal{R}_0=id$ and consequently with the same arguments we conclude that $\beta(v)=v$ for all $v\in V_{k+1}\setminus V_k$ where $k<i$.
\end{proof}

\section{The Urysohn space}

In this section by an \emph{embedding} we always mean an isometric embedding. Given a metric space $X$, its metric will be denoted by $d$; the distance from a point $x$ to a set $S$ will be denoted by $d(x,S)$.
Recall that the \emph{Urysohn space} $\mathbb{U}$ is a separable complete metric space satisfying the following property
\begin{itemize}
\item[(E)] For every finite metric spaces $E\subseteq F$ and for every embedding $e:E\to \mathbb{U}$ there exists an embedding $f:F\to \mathbb{U}$ such that $f\upharpoonright E=e$.
\end{itemize}
As in the graph case, this property characterizes $\U$ up to isometries and it implies homogeneity with respect to finite metric spaces.

\begin{definition}
We say that a metric space $X$ of density $\omega_1$ is \emph{Urysohn-like} space iff it is complete and can be represented as the colimit of a chain $(\mathbb{U},e_i)_{i<\omega_1}$.
\end{definition}

The following fact is easy to prove, by a simple closing-off argument.

\begin{proposition}
A complete metric space $X$ is Urysohn-like if and only if it has density $\omega_1$ and satisfies condition (E). 
\end{proposition}

As in the graph case it is known that there is Kat\v etov functor for metric spaces so we have.

\begin{corollary}
There exists a homogeneous Urysohn-like space.
\end{corollary}

Our goal is to construct an Urysohn-like space $X$ that is rigid. We want to prove similar results as in the Rado-like case. We may use a similar strategy as in the graph case to prove that there is an embedding $e:\mathbb{U}\to\mathbb{U}$ which does not extend any non-trivial automorphism (recall that an isomorphism is a bijective isometry).
Roughly speaking, we shall add a special point $x$ to $\mathbb{U}$ and then fill the space such that we obtain again an isometric copy of $\mathbb{U}$ in such a way that this special point $x$ must be preserved by every automorphic extension. Since the Urysohn space has no isolated points, we must assure ourselves that after filling the space with some countable dense part to obtain again $\mathbb{U}$ there will be no point in its closure with similar properties as $x$ has. 

For a metric space $X$ we denote its metric extension by adding a set $\{x_i\}_{i\in I}$ as $X\oplus_{i\in I}x_i$. This notation means that not only we add the points but we have already chosen a metric on the new space. We denote the metric on all spaces by $d$ because it is always clear from the context which space we mean.
The \Fra\ limit of all finite rational metric spaces is denoted by $\mathbb{Q}\mathbb{U}$. It is well-known that $\overline{\mathbb{Q}\mathbb{U}}=\mathbb{U}$. 

\begin{definition}
Given a positive $r\in\mathbb{R}$ define $\mathcal{M}_r$ to be the category of finite metric spaces with the following property. Objects are spaces of the form $E\oplus x$ where $E$ is a finite metric space, $x$ is a special distinguished point and $d(x,y)\ge r$ for all $y\in E$. Isometric embeddings $f:E\oplus x\to F\oplus x'$ are morphisms in $\mathcal{M}_r$ provided that $f(x)=x'$. We denote by $\mathbb{Q}\mathcal{M}_r$ the subcategory of $\mathcal{M}_r$ consisting of all rational metric spaces.
\end{definition}

We always denote the special point by $x$. First observation is that $\mathbb{Q}\mathcal{M}_r$ is a \Fra\  class since it has amalgamations similar as in the category of metric spaces. The next definition is crucial, as it distinguishes continuum many different one point extensions of $\mathbb{U}$. 

\begin{definition}
For $0<r\in\mathbb{R}$ we say that the one point extension $\mathbb{U}\oplus x$ of the Urysohn space is \emph{$r$-Urysohn} (or \emph{$r$-Urysohn type} or simply \emph{$r$-type}) iff the following holds
\begin{itemize}
\item $d(x,\mathbb{U})=r$,
\item for every pair $E\oplus x\subseteq F\oplus x\in\mathcal{M}_r$ and for every embedding $e:E\oplus x\to \mathbb{U}\oplus x$ such that $e(x)=x$ there is an embedding $f:F\oplus x\to \mathbb{U}\oplus x$ such that $f\upharpoonright E\oplus x=e$.
\end{itemize}
\end{definition}

\begin{observation}
For every $0<r\in\mathbb{R}$ there is an r-type and it is unique up to isomorphisms (i.e., isometries preserving the special point).
\end{observation}
\begin{proof}
We prove that the $r$-type $\mathbb{U}\oplus x$ is the completion of the \Fra\ limit of $\mathbb{Q}\mathcal{M}_r$. It is the same argument as in the proof of $\overline{\mathbb{Q}\mathbb{U}}=\mathbb{U}$, because the \Fra\  limit of $\mathbb{Q}\mathcal{M}_r$ has the form $\mathbb{Q}\mathbb{U}\oplus x$.

Uniqueness can be proved by the back-and-forth argument when we fix a countable dense set in each space as in the case of proving uniqueness of $\mathbb{U}$.
\end{proof}

\begin{observation}
Let $e:\mathbb{U}\hookrightarrow \mathbb{U}$ be an isometric embedding and let $x$ be a realization of some $r$-type over $e(\mathbb{U})$. Assume that we have a pair $\alpha,\beta$ where $\alpha:\mathbb{U}\to \mathbb{U}$ is an automorphism and $\beta$ extends $\alpha$ via $e$. Then $\beta(x)$ is again an $r$-type over $e(\mathbb{U})$.
\end{observation}

We need to define another type of one point extensions of $\mathbb{U}$ which will fill-out the space making it again Urysohn in such a way that no limit point of a sequence of this types is an $r$-type.

\begin{definition}
For an arbitrary metric space $X$ we say that the one point extension $X\oplus x$ \emph{has finite support} over $X$ (or it is a \emph{finitely supported type}, or $x$ \emph{realizes a finitely supported type}) iff there is a finite set $Y\subseteq X$ such that
$$d(x,z)=\inf\{d(x,y)+d(y,z):y\in Y\}.$$
\end{definition}

It can be easily shown that the formula above  is a correct definition of a metric. 

\begin{lemma}
Assume that we have a space $\mathbb{U}\oplus_{i\in\mathbb{N}} x_i\oplus y$
where $y$ realizes a type with finite support over $\mathbb{U}\oplus_{i\in\mathbb{N}} x_i$. Then $y$ does not realize any $r$-type for $0<r\in \mathbb{R}$ over $\mathbb{U}$.
\end{lemma}
\begin{proof}
We will construct a finite space $E\oplus w$ which is not realized in $\mathbb{U}\oplus y$ when sending $w$ to $y$. Moreover $d(w,E)$ can be arbitrarily big, which means that it is not an $r$-type for any $r\in\mathbb{R}$. Assume that there are $n$ points $\{a_i\}_{i\le n}$ realizing the finite support of $y$. Fix a number $L$ bigger than any $d(y,a_i)$. 

Let the universe of the space $E$ be formed by $n+2$ distinct points $\{w_i\}_{i\le n+1}\cup \{w\}$. Let the metric be as follows
\begin{itemize}
\item for $i\not=j$ set $d(w_i,w_j)=2L$,
\item $d(w,w_i)=L$.
\end{itemize} 
It is easy to see that this is actually a metric space. We claim that it cannot be realized in the sense described before. Suppose it is. So there is a mapping $f:E\oplus w\to \mathbb{U}\oplus y$ such that $f(w)=y$. Denote $f(w_i)$ as $y_i$. By the pigeon hole principle there are $i\not=j$ and $k\le n$ such that $d(y,y_i)=d(y,a_k)+d(a_k,w_i)$ and $d(y,y_j)=d(y,a_k)+d(a_k,y_j)$. But $\mathbb{U}\oplus_{i\in\mathbb{N}} x_i\oplus y$ is a metric space so we must have $2L=d(y_j,y_i)\le d(y_j,a_k)+d(a_k,y_i)=2L-2n_k$, which is a contradiction.
\end{proof}

A natural question is how far can the closest $r$-type realized by $z$ be from our fixed point $y$ of finite support. Imagine that there is a mapping $f:E\oplus w\to\mathbb{U}\oplus z$ such that $e(w)=z$. By the triangle inequality and an argument leading to a contradiction in the previous proof, we have for the contradicting pair $w_j,w_i$ and $k\le n$ that
$$d(z,y)+L=d(z,y)+d(z,f(w_i)\ge d(y,f(w_i))=d(y,a_k)+d(a_k,f(w_i)).$$
Again by the triangle inequality 
$$d(a_k,f(w_j))+d(a_k,f(w_i))\ge d(f(w_j),f(w_i))=2L$$
therefore $\max\{d(a_k,f(w_i)),d(a_k,f(w_j))\}\ge L$. All these together gives rise to
$$d(z,y)+L\ge d(y,a_k)+\max\{d(a_k,y_i),d(a_k,y_j)\}\ge d(y,a_k)+L.$$

\begin{lemma}
Assume that we have a space $\mathbb{U}\oplus_{i\in\mathbb{N}} x_i\oplus y$
where $y$ realizes a type with finite support over $\mathbb{U}\oplus_{i\in\mathbb{N}} x_i$.
Then it is not an $r$-type over $\mathbb{U}$ for any $r\in \mathbb{R}$. Moreover, every realization $z$ of some $r$-type must satisfy $d(z,y)\ge d(y,\mathbb{U}\oplus_{i\in\mathbb{N}} x_i)$. 
\end{lemma}
\begin{proof}
The additional part follows from the above discussion and
$$d(z,y)\ge d(y,a_k)\ge \min_{i\le n}\{d(y,a_i)\}= d(y,\mathbb{U}\oplus_{i\in\mathbb{N}} x_i).$$
\end{proof}

Assume that we have a separable metric space $X$. We want to enlarge it to $X\oplus_{i<\omega}x_i$ such that $\overline{X\oplus_{i<\omega}x_i}\simeq \mathbb{U}$. There is a construction using so-called Kat\v etov maps. There is a natural metric on the space $tp(X)$ of all one point extensions of the metric space $X$, but this space may no longer be separable. The Kat\v etov construction adds to $X$ the subspace of $tp(X)$ which is generated by all extensions with finite support. We denote it by $X^1$ and put $X^{n+1}:=(X^n)^1$. It is not surprising that $\overline{\bigcup_{n<\omega}X^n}\simeq \mathbb{U}$. Because the subspace of $tp(X)$ generated by types of finite support is separable, it suffices to choose only countable many such types and add all of them. After iterating this process we may re-enumerate these types in such a way that $\overline{X\oplus_{i<\omega}x_i}\simeq \mathbb{U}$ and $x_n$ has finite support over $X\oplus_{i<n} x_i$. So in fact for every separable metric space there is a sequence of finitely supported types, but with the support on the previously defined ones, turning it to $\mathbb{U}$. This can be described more directly. Fix a countable dense set $Y \subseteq X$ and build a space $Y\oplus_{i<\omega} y_i$ such that $y_n$ has rational finite support over $Y\oplus_{i<n} y_i$.
We have to make sure that we eventually use all such types. That is possible since there are only countably many possibilities. In the next definition we just state what we mean by iterated finite support which will be needed as the previous comment suggests. 

\begin{definition}
Let $X$ be a metric space and consider its extension $X\oplus_{i\le n}y_i$ such that $y_k$ has finite support over $X\oplus_{i<k}y_i$ for all $k\le n$. We assign to $y_n$ a finite subset $s_{y_n}(X)\subseteq X$ such that for all $y\in s_{y_n}(X)$ there is a finite subsequence $\{y_{i_k}\}_{k\le m}$ with $y_n=y_{i_0}$, $y=y_{i_m}$ and $y_{i_{k+1}}$ is in the support of $y_{i_k}$ for all $k<m$.
We denote all such subsequences associated to a point $y$ by $seq_{y_n\to y}(X)$. We define a function $d_{y_n}:Y\to \mathbb{R}$ by
$$d_{y_n}(y):=\min\left\{\sum_{k=0}^{m-1} d(y_{i_k},d_{i_{k+1}}):\{y_{i_k}\}_{k\le m}\in seq_{y_n\to y}(X) \ \right\}.$$
\end{definition}

The set $s_{y_n}(X)$ is finite and for $z\in X$ we can compute any distance $d(y_n,z)$ by
$$d(y_n,z)=\min\{d_{y_n}(y)+d(y,z):y\in s_{y_n}(X)\}$$
because it is just an iteration of finite support. Notice that $d(y_n,y)\le d_{y_n}(y)$ for all $y\in supp_{y_n}(X)$ but the equality does not hold in general. In fact there is a subset of $s_{y_n}(X)$ for which it does hold and $d(y_n,z)$ can be calculated using this subset only. Let us denote this subset by $supp_{y_n}(X)$. Then the type which $y_n$ realizes over $X$ is finitely supported on the $supp_{y_n}(X)$ with distances described above.

\begin{lemma}
Let $X=(\mathbb{U}\oplus_{j<\omega}x_j)$ be a metric space. Assume that we have an extension $X \oplus_{i\le n}y_{i}$ such that $y_k$ has finite support over $X\oplus_{i<k}y_i$ for all $k\le n$ and assume that we add a point $z$ realizing an $r$-type over $\mathbb{U}$. Than we have that
$$d(y_n,z)\ge d(y_n,\mathbb{U}\oplus_{j<\omega}x_{j}).$$
\end{lemma}
\begin{proof}
This follows immediately from the previous lemma because, due to the discussion above, $y_n$ is in fact finitely supported over $X$.
\end{proof}  

Before we prove the next lemma, recall that the Urysohn $r$-type is $\mathbb{U}\oplus x=\overline{\mathbb{Q}\mathbb{U}\oplus x}$, where $\mathbb{Q}\mathbb{U}\oplus x$ is the \Fra\ limit of $\mathbb{Q}\mathcal{M}_r$.

\begin{lemma}
For a sequence $\{r_i\}_{i<\omega}$ with $r_i>1$ there is an extension of the Urysohn space $\mathbb{U}\oplus_{i<\omega}x_i$ with following properties
\begin{itemize}
\item $x_i$ realizes each $r_i$-type,
\item $d(x_i,x_j)>1$ for $i\not=j$,
\item for every pair of points $\{x,y\}\in [\mathbb{Q}\mathbb{U}]^2$ there is $i<\omega$ such that $d(x,x_i)+d(x,y)=d(x_i,y)$.
\end{itemize}
\end{lemma}
\begin{proof}
Fix an enumeration of all pairs $\{x,y\}\in [\mathbb{Q}\mathbb{U}]^2$. We denote it by $\{p_k\}_{k<\omega}$ and use $p_k=\{p_k^0,p_k^1\}$. 

Next we proceed by induction. Assume that we have already constructed $\mathbb{U}\oplus_{i<n}x_i$ with given properties and we want to extend it to a point $x_n$. We describe the first element of a sequence in the category $\mathbb{Q}\mathcal{M}_{r_i}$ 
We denote it by $E\oplus x$. We choose $E=\{p_n^0,p_n^1\}\subseteq \mathbb{Q}\mathbb{U}$ in such a way that we can find a metric on $E\oplus x$ such that $d(x,p_n^0)+d(p_n^0,p_n^1)=d(x,p_n^1)$. Then we take the closure of the \Fra\  limit $\mathbb{U}\oplus x_n$ and embedd it in the space $\mathbb{U}\oplus_{i<n}x_i$ in a proper way i.e. mapping $E$ from the limit to $E\subseteq \mathbb{Q}\mathbb{U}$. Then take the amalgamation given by the push-out (i.e., the maximal amalgamation). We are done, because $d(x_n,x_i)\ge r_n+r_i$.
\end{proof}

\begin{theorem}\label{embedding}
There is an embedding $e:\mathbb{U}\to\mathbb{U}$ such that $G_e=\{id\}$.
\end{theorem}
\begin{proof}
For a fixed sequence $\{r_i\}_{i<\omega}$ with $r_i>1$ take $\mathbb{U}\oplus_{i<\omega}x_i$ as in the previous lemma. We use the argument described in a discussion above to create a dense set of the Urysohn space of the form $\mathbb{Q}\mathbb{U}\oplus_{i<\omega}x_i\oplus_{j<\omega}y_j$, where $y_j$ is a finite type over $\mathbb{Q}\mathbb{U}\oplus_{i<\omega}x_i\oplus_{k<j}y_k$. We claim that after taking $\overline{\mathbb{Q}\mathbb{U}\oplus_{i<\omega}x_i\oplus_{j<\omega}y_j}\simeq \mathbb{U}$ the only $r$-types over the original $\mathbb{U}$ are realized by $\{x_i\}_{i<\omega}$. Assume that there is $z$ which realizes an $r$-type and $d(z,\mathbb{U}\oplus_{i\in\mathbb{N}}x_i)>\epsilon$. So there is a point $y_k$  which realizes a finitely supported type over $\mathbb{U}\oplus_{i\in\mathbb{N}}x_i\oplus_{j<k}y_j$ such that $d(y_k,z)<\frac{\epsilon}{4}$ and this is impossible due to the previous lemma, because
$$\frac{\epsilon}{4}>d(y_k,z)\ge d(y_k,\mathbb{U}\oplus_{i\in\mathbb{N}}x_i)\ge d(z,\mathbb{U}\oplus_{i\in\mathbb{N}}x_i)-d(y_k,z)=\epsilon -\frac{\epsilon}{4}.$$

Suppose that we have an automorphism $\alpha\in \Aut(\mathbb{U})$ and we can extend it by some $\beta\in \Aut(\mathbb{U})$ via $e$. Because for every $r_i$ there is exactly one $r_i$-type in the extended $\mathbb{U}$, we have that $\beta(x_i)=x_i$. Assume that for some $z\in \mathbb{U}$ we have $\alpha(z)=y\not=z$. Due to the construction of $\{x_i\}_{i<\omega}$, there are always $k<\omega$ and points $z',y'\in\mathbb{Q}\mathbb{U}$ close enough to $y,z$ such that $d(x_k,z')=d(z',y')+d(x_k,y')$ contradicting $d(x_k,z)=d(\beta(x_k),\beta(z))=d(x_k,y)$ which should hold because $\beta$ is an isometry.
\end{proof}

We need a generalization of the previous lemma to prove the main theorem. In fact, we need to generalize it for the situation when we iterate by adding arbitrary and finitely supported types together. To be more concrete, let us describe such a situation. Assume that we have a finite extension of a space $X_0$ by some points with finite support over $X_0$, i.e., we set $Y_0:=X_0\oplus_{i<n_1}y_{0,i}$. Then we add finitely many points with arbitrary support over $Y_0$, i.e., $X_1:=Y_0\oplus_{i<k_1}x_{0,i}$. We repeat this procedure $m$ times obtaining
$$X:=X_0\oplus_{j<m,i<k_j}x_{j,i}\oplus_{p<m,q<n_p}y_{p,q}.$$
For every one point extension $X\oplus y$ with finite support we may define $s^*_y(X)$ which is a subset of $X_0\oplus_{j<m,i<k_j}x_{j,i}$ defined by the following: $z\in s^*_y(X)$ iff there exists a sequence of pairs of numbers $\{p_i,q_i\}_{i<m'}$ strictly decreasing in the first coordinate and such that $y_{p_{i+1},q_{i+1}}$ is in the support of $y_{p_{i},q_{i}}$ (where by \emph{support} we mean its finite support over $X_{p_{i}}$), $y_{p_0,q_0}$ is in the support of $y$, and $z$ is in the support of $y_{p_{m'-1},q_{m'-1}}$.

Similarly as in the previous case, we may define a function $d^*_{y}:s^*_y(X)\to \mathbb{R}$, where $d^*_{y}(z)$ is the minimal sum of distances over all sequences going from $y$ to $z$ as described above. This is a generalization of the previous definition, because we can calculate distances only for $x\in X_0$ by 
$$d(y,x)=\min\{d_{y}(z)+d(z,x):z\in s^*_{y}(X)\}$$
which can be verified similarly as before. 

\begin{lemma}
Consider a space $X:=(\mathbb{U}\oplus_{j<m,i<k_j}x_{j,i}\oplus_{p<m,q<n_p}y_{p,q})\oplus y$ described in the previous paragraph. Then for every point $z$ realizing an $r$-type over $\mathbb{U}$ we have that $d(y,z)\ge d(y,\mathbb{U}\oplus_{j<m,i<k_j}x_{j,i})$. In particular, $y$ does not realize any $r$-type over $\mathbb{U}$.
\end{lemma}
\begin{proof}
We only need to observe that the type of $y$ over $\mathbb{U}$ is in fact finitely supported, however its support $s^*_y(X)$ may not be contained in $\mathbb{U}$. But the proof of the previous cases requires only points from $\mathbb{U}$ to obtain the inequality so we may use it as well to obtain
$$d(y,z)\ge \min\{d^*_y(z),z\in s^*_y(X)\}\ge d(y,\mathbb{U}\oplus_{j<m,i<k_j}x_{j,i}).$$
\end{proof}

\begin{theorem}
There is a rigid Urysohn-like space of density $\aleph_1$.
\end{theorem}
\begin{proof}

We proceed  by induction of length $\omega_1$. Fix a matrix $\{r_{\alpha,j}\}_{\alpha<\omega_1,j<\omega}$ of pairwise different real numbers, where $r_{\alpha,j}>1$ such that $r_{\alpha,j}\to\infty$ as $j\to \infty$. Next we define a transfinite sequence of spaces and their countable dense parts as follows. Put $X_0:=\mathbb{U}$ and $X^{'}_0:=\mathbb{Q}\mathbb{U}$. For $\alpha$ limit let $X_{\alpha}:=\overline{\bigcup_{\beta<\alpha}X_{\beta}}$ and $X^{'}_{\alpha}:=\bigcup_{\beta<\alpha}X^{'}_{\beta}$. Assume that we have already constructed $\{X_{\alpha}\}_{\beta\le \alpha}$ and that $\mathbb{U}\simeq X_{\beta}$ for all $\beta\le\alpha$. Given $\alpha+1$, first enumerate all tuples of $X'_{\alpha}$ by $\{p_j\}_{j<\omega}$ in such a way that $d(p_j^0,p_j^1)\le r_{\alpha,j}$. This is possible since $r_{\alpha,j}\to\infty$ as $j\to\infty$. Put $X_{\alpha}\oplus_{j<\omega}x_{\alpha,j}$, where $x_{\alpha,j}$ realizes an $r_{\alpha,j}$-type over $X_{\alpha}$ and separate a tuple $p_j$ as in the proof of Theorem~\ref{embedding}.

Moreover this can be done in such a way that for each $x_{\alpha,j}$ there is $y\in X_0$ with $d(y,x)=r_{\alpha,j}$. We use homogeneity of $\mathbb{U}$ here. We may always find a triangle $\{w,p_j^0,p_j^1\}\subseteq X_{\alpha}$ such that $w\in X_{0}$ and $d(w,p_j^0)\ge d(w,p_j^1)>r_{\alpha,j}$. Let us define the four-point metric space $\{z,q_j^0,q_j^1,x_{\alpha,j}\}$ by the following conditions.
\begin{itemize}
\item $d(z,q_j^l)=d(w,p_j^l)$,  $d(q_j^0,q_j^1)=d(p_j^0,p_j^1)$,
\item $d(x_{\alpha,j},z)=r_{\alpha,j}$,  $d(x_{\alpha,j},q_j^0)=d(z,q_j^1)$ and $d(x_{\alpha,j},q_j^1)=d(z,q_j^1)+d(q_j^1,q_j^1)$.
\end{itemize}
One can easily verify that this is indeed a metric space. Now once we have $\mathbb{U}\oplus x_{\alpha,j}$ and $x_{\alpha,j}$ realizes an $r_{\alpha,j}$-type then it must contain our just-defined four-point space. By the homogeneity of $\mathbb{U}$ we may assume that $z=w$ and $q_j^l=p_j^l$. Finally, fill-out the space  $X^{'}_{\alpha}\oplus_{j<\omega}x_{\alpha,j}$ by finitely supported types $\{y_i\}_{i<\omega}$ as in the proof of Theorem~\ref{embedding}, in order to obtain $X_{\alpha+1}:=\overline{X_{\alpha}\oplus_{j<\omega}x_{\alpha,j}\oplus_{i<\omega}y_i}\simeq\mathbb{U}$ and define $X^{'}_{\alpha+1}:=X^{'}_{\alpha}\oplus_{j<\omega}x_{\alpha,j}\oplus_{i<\omega}y_i$. Finally, let $X:=\bigcup_{\alpha<\omega_1}X_{\alpha}$.

Fix $\beta<\omega_1$.
We claim that all $r$- types over $X_\beta$ are in the set
$$Y:=X_\beta\oplus_{\beta\le\alpha<\omega_1,j<\omega}x_{\alpha,j}.$$
Suppose that there is $z\in X$ which realizes an $r$-type over $X_\beta$ and $d(z,Y)>\epsilon$. There must be $\alpha<\omega_1$ such that $z\in X_{\alpha}$. We find $t\in X^{'}_{\alpha}$ with $d(t,z)<\frac{\epsilon}{4}$ which has a finite support over some $X^{'}_{\gamma}\oplus_{j<\omega}x_{\gamma,j}\oplus_{l<n}y_l$ where $\gamma<\alpha$. If we consider $s^*_t(X^{'}_\beta\oplus_{\beta\le\zeta\le \gamma,j<\omega}x_{\zeta,j})$ then we suddenly are in the same situation as in one of the previous lemmas, because $|s^*_t(X^{'}_\beta\oplus_{\beta\le\zeta\le \gamma,j<\omega}x_{\zeta,j})|<\omega$. Thus we obtain again a contradiction:
$$\frac{\epsilon}{4}>d(t,z)\ge d(t,X_\beta\oplus_{\beta\le\zeta \le \gamma,j<\omega}x_{\zeta,j})\ge d(z,X_\beta\oplus_{\beta\le\zeta \le \gamma,j<\omega}x_{\zeta,j})-d(t,z)\ge\epsilon -\frac{\epsilon}{4}.$$

The matrix $\{r_{\alpha<\omega_1,j<\omega}\}$ contains pairwise different real numbers and every automorphism $f:X\to X$ must be invariant on club many $X_\beta$.
More precisely, there is a closed unbounded set $C \subseteq \omega_1$ such that for $\beta \in C$ the restriction of $f$ to $X_\beta$ is an automorphism of $X_\beta$. Fix $\beta \in C$. We have $f(x_{\beta,j})=x_{\beta,j}$ for all $j<\omega$, because $x_{\beta,j}$ realizes $r_{\beta,j}$-type and it cannot be moved since it is the only one in $X$. The reason is that we have for each $x_{\alpha,j}$ some point $y\in X_0$ such that $d(x_{\alpha,j},y)=r_{\alpha,j}$ due to the construction, so if among the set $\{x_{\alpha,j}\}_{\beta<\alpha<\omega_1,j<\omega}$ there is some $r$-type over $X_{\beta}$ then it must respect its value $r_{\alpha,j}$, because we have $d(x_{\alpha,j},X_{\beta})=r_{\alpha,j}$. Due to similar construction of $X_{\beta+1}$ as in Theorem~\ref{embedding}, $f$ has to be identity on $X_\beta$, therefore it has to be identity on $X$.
\end{proof}

\section{Final remarks}

Let us mention here what can be said about the universality of objects that we have constructed. This was kindly suggested by the referee. The concrete question may be as follows: Is there a rigid Urysohn-like metric space of density $\aleph_1$ which is universal for all metric spaces of density $\aleph_1$? The answer may be positive if there exists a universal metric space of density $\aleph_1$ but it is known that such a space may not exist (for example in the Cohen model). On the other hand, it is not known whether it is consistent that such a space exists under the failure of CH (this is known in case of graphs, see~\cite{shelah}). Every metric space of density $\aleph_1$ can be embedded in some Urysohn-like space of density $\aleph_1$, this can be achieved simply by using the \Kat\ functor. So the best that one can get is the following conjecture.

\begin{conjecture}
For every metric space $M$ of density $\aleph_1$ there exists some rigid Urysohn-like space $X$ of density $\aleph_1$ such that $M$ can be embedded in $X$. 
\end{conjecture}

In another words the conjecture states that the class of all rigid Urysohn-like spaces of density $\aleph_1$ is universal for all metric spaces of density $\aleph_1$. The conjecture implies that if there exists a universal metric space of density $\aleph_1$ then there exists one which is rigid Urysohn-like.

Another interesting problem is to get a similar result in the class of finite dimensional Banach spaces $\mathcal{B}$ where embeddings are linear isometric monomorphisms.
The \Fra\ limit is so called \emph{\Gur\ space} $\mathbb{G}$.
It is characterised as the unique separable Banach space such that $Age(\mathbb{G})=\mathcal{B}$ and it has the approximate extension property with respect to $\mathcal{B}$ (where the distance between two maps is the norm of their difference).

\begin{question}
Is there an embedding $e:\mathbb{G}\to\mathbb{G}$ such that $G_e=\{id,-id\}$? Is there a rigid \Gur -like Banach space?
\end{question}

The class $\sigma\mathcal{B}$ of separable Banach spaces admits a \Kat\ functor which was proved in~\cite{Yaacov1}. 
The strategy to answer the question may be the same as in the case of metric spaces i.e. find some special one point extensions that cannot be approximated by finitely supported one point extensions and then use the \Kat\ functor.

\bibliographystyle{amsplain}

\end{document}